\documentclass{article}

\usepackage{amssymb, amsmath,amsthm, tikz}
\newtheorem{theorem}{Theorem}
\newtheorem{corollary}{Corollary}
\newtheorem{remark}{Remark}
\newtheorem{lemma}{Lemma}

\author{Alina Pavlikova\footnote{\noindent 
UNIGE, Villa Battelle, 1227 Carouge, Suisse\hfill\break\ \ \ \ \ \  kussike(at)gmail.com} }
\title{Critical values of coamoebas\\of codimension two affine planes}
\date{}

\begin{document}
\maketitle
\begin{abstract}Despite some general results about (co)amoebas of half-dimensional varieties and linear spaces (see [1,2,3]), very little is  known about the topological structure of the corresponding critical loci. In this note, we give a preliminary discription of critical values for the coamoeba of a generic affine plane in $\mathbb{C}^4$.
\end{abstract}

Consider a generic affine plane $V$ in $\mathbb{C}^4$.  Denote by $V^\times$ the intersection of $V$ with the algebraic torus $(\mathbb{C}^\times)^4$ (recall that $\mathbb{C}^\times=\mathbb{C}\backslash\{0\}$). There are maps 
$Log:(\mathbb{C}^\times)^4\rightarrow\mathbb{R}^4,\ 2Arg:(\mathbb{C}^\times)^4\rightarrow(\mathbb{R}P^1)^4$ and $Arg:(\mathbb{C}^\times)^4\rightarrow(S^1)^4$. Let $\mathcal{A}$, $\mathcal{B}$ and $\mathcal{C}$ be the restrictions of these maps to $V^\times,$ we call them amoeba, rolled coamoeba and (usual) coamoeba map respectively. Note that their loci of critical points are identical -- denote by $Z$ this locus. Fix the notations $\mathcal{A}_Z$, $\mathcal{B}_Z$, $\mathcal{C}_Z$ for the critical values of $\mathcal{A}$, $\mathcal{B}$, $\mathcal{C}.$

\begin{theorem}
 The space $\mathcal{B}_Z$ is homeomorphic to $\mathbb{R}P^2$ blown-up at four points.
\label{thm_rolledcritical}
\end{theorem}

\begin{theorem}
There is a compactification $\hat{Z}$ of $Z$ -- a total space of a locally trivial circle bundle over $\mathcal{B}_Z$ with projection map being an extention of $\mathcal{B}|_Z.$ There are five distiguished sections of $\hat{Z}\rightarrow\mathcal{B}_Z$ such that $Z$ is the result of removing these sections from $\hat{Z}.$
\label{thm_five}
\end{theorem}
\begin{corollary}
$Z$ is a nonsingular 3-manifold.
\label{cor_3mfld}
\end{corollary}
\begin{theorem}
The projection from $Z$ to $\mathcal{C}_Z$ is a locally trivial fiber bundle with open intervals as its fibers. 
\label{thm_fiber}
\end{theorem}
\begin{remark}
Removing the five sections from a fiber-circle of $\hat{Z}\rightarrow\mathcal{B}_Z$ we obtain a union of three-to-five intervals, fibers of $Z\rightarrow\mathcal{C}_Z.$ 
\label{rem_five}
\end{remark}
\begin{remark}
The sixteen-fold covering from $(S^1)^4$ to $(\mathbb{R}P^1)^4$ gives a one-to-one correspondense on the regular values of $\mathcal{B}$ and $\mathcal{C}$. Its restriction $\mathcal{C}_Z\rightarrow\mathcal{B}_Z$ is generically five-to-one.
\label{rem_sixteen}
\end{remark}

\begin{theorem}
There is a compactification $\hat{\mathcal{C}}_Z$ of $\mathcal{C}_Z$ homeomorphic to $\mathbb{R}P^2$ blown-up at sixteen points  and  $\hat{\mathcal{C}}_Z\backslash\mathcal{C}_Z$ is a union of ten disjoint topological circles. Moreover, $\mathcal{C}_Z\rightarrow\mathcal{B}_Z$ extends to a five-fold unramified covering $\hat{\mathcal{C}}_Z\rightarrow\mathcal{B}_Z.$
\label{thm_ten}
\end{theorem}

\section*{Proof of Theorem \ref{thm_rolledcritical}}
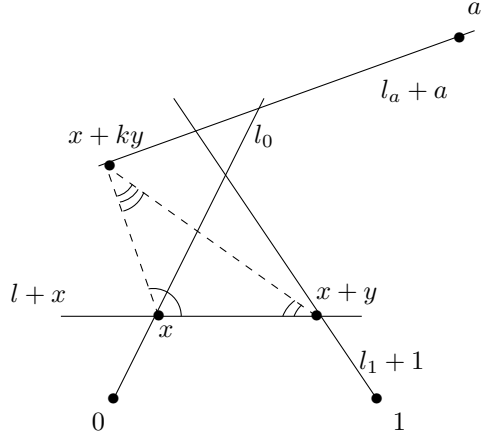
\begin{figure}[htbp]
\centering
\begin{tikzpicture}
\draw(0.2,0)--++(2,4);
\draw(1,4)--++(2.7,-4);
\draw (-0.8,1.4) node {$l+x$};
\draw (0.9,0.9) node {$x$};
\draw (0.8,1.1) node {$\bullet$};

\draw (5,5.2) node {$a$};
\draw (4.2,4.1) node {$l_{a}+a$};
\draw (4.8,4.8) node {$\bullet$};
\draw (5,4.9)--++(-5,-1.8);
	
\draw (4,-0.3) node {$1$};
\draw (3.7,0) node {$\bullet$};
\draw (0,-0.3) node {$0$};
\draw (0.2,0) node {$\bullet$};
\draw [dashed](0.8,1.1)--++(-0.7,2);
\draw [dashed](2.9,1.1)--++(-2.8,2);
\draw (2.9,1.1) node {$\bullet$};
\draw (3.3,1.4) node {$x+y$};
\draw(-0.5,1.1)--++(4,0);
\draw (0.15,3.1) node {$\bullet$};
\draw (0.1,3.5) node {$x+ky$};
\draw (3.9,0.5) node {$l_{1}+1$};
\draw (2.2,3.5) node {$ l_{0}$};

\draw (1.1,1.1) arc (0:100:0.37cm);

\draw (0.25,2.7) arc (-77:-27:0.30cm);
\draw (0.285,2.62) arc (-80:-30:0.35cm);
\draw (0.33,2.5) arc (-75:-20:0.40cm);

\draw (2.6,1.1) arc (160:132:0.38cm);
\draw (2.45,1.1) arc (160:125:0.44cm);

 \end{tikzpicture}

\caption{The rolled coamoeba map $\mathcal{B}(x,y,x+y-1, x+ky-a)=(l_0,l,l_1,l_a).$}
\label{linkage}
\end{figure}

Without loss of generality we assume that $V$ is parametrized as 
$$(x,y) \mapsto (x,y,x+y-1, x+ky-a)$$
for generic $a,k\in\mathbb{C}$. This allows to draw the associated linkage (see Figure \ref{linkage}).

From now on we denote by $(l_0,l,l_1,l_a)$ a point in $(\mathbb{R}P^1)^4,$ where we think of every of the four lines as passing throug the origin in $\mathbb{C}.$ Consider $\mathbb{R}P^2$ as the compactification of $\mathbb{C}$ by the line $\mathbb{R}P^1$ at infinity. 
\begin{lemma}
If $(l_0,l,l_1,l_a)\in \mathcal{B}_Z$ then there exist a unique point $p\in\mathbb{R}P^2$ such that the closures of $l_0, l_1+1, l_a+a \subset \mathbb{C}$ intersect at $p.$ \label{lemma_pointp}
\end{lemma}
Therefore, we have a map $\phi:\mathcal{B}_Z\rightarrow\mathbb{R}P^2$ assigning the intersection point $p$ to a critical value. It is easy to see that $\phi$ is birational.

\begin{remark}
For any $l \in \mathbb{R}P^1$ the locus $$C_l = \{p \in \mathbb{R}P^2 | \exists (l_0, l, l_1, l_a) \in \mathcal{B}_Z:\, \phi(\ell_0, \ell, \ell_1, \ell_a) = p\}$$ is a conic curve. Moreover, the map $l \in \mathbb{R}P^1 \mapsto C_l \in \mathbb{R}P^5$ is linear.
\end{remark}
\noindent In other words, $\{C_l\}_{l\in{\mathbb{R}P^1}}$ is a pencil of conics. 

\begin{figure}
\includegraphics{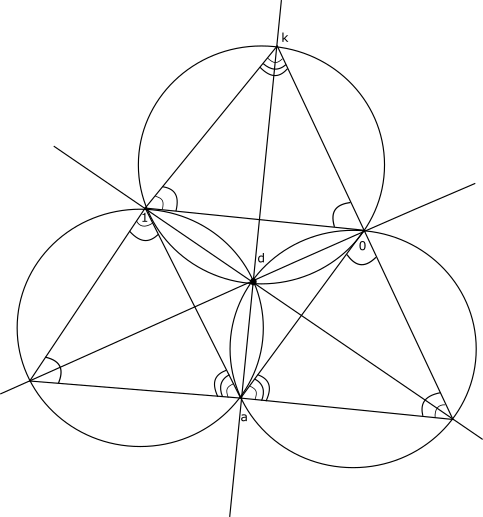}
\caption{Constructing $d,$ the fourth base point of the pencil of conics. Every circle corresponds to a locus of points $p$ where all the triangles in the fiber of the rolled coamoeba map over $\mathcal{B}_Z$ have $p$ as one of the three vertices. For example, the circle passing through $0,1,k$ corresponds to the case when $\lambda(l_0,l,l_1)=\{p\}$ in the notations of the proof of Lemma \ref{lemma_pointp}. Other two circles are obtained by symmetry between points $0,1,a$ (see Reamrk \ref{rem_sym}).}
\label{pointD}
\end{figure}

\begin{remark} The base points of this pencil are $0,\ 1,\ a$ and $d,$ where the fourth point is constructed on Figure \ref{pointD}.
\end{remark}

\noindent Denote by $\tilde{C}=\{(p,l)\in\mathbb{R}P^2\times\mathbb{R}P^1|\,p\in C_l\}$ the universal curve of this pencil.

\begin{remark}
$\tilde{C}$ is a blow-up of $\mathbb{R}P^2=\mathbb{C}\cup\mathbb{R}P^1$ at four points $0,\ 1,\ a,\ d\in\mathbb{C}$.
\end{remark}

\noindent Finally, the map $\phi$ extends to a homeomorphism $\tilde{\phi}:\mathcal{B}_Z \to \tilde{C}$ given by $$\tilde{\phi} (l_0,l,l_1,l_a)=(\phi(l_0,l,l_1,l_a), l).$$

\subsection*{Proof of Lemma \ref{lemma_pointp}}

\begin{lemma}

The map $\mathcal{B}$ is rational.

\end{lemma}

\begin{proof}

Note that $2arg(z)=[{Re(z)}:{Im(z)}]\in\mathbb{R}P^1$, for a non-zero complex number $z$, and $\mathcal{B}$ is a composition of the coordinatewise $2arg$ and the affine embedding of $V$ to $\mathbb{C}^4$.

\end{proof}

\begin{lemma}

The map $\mathcal{B}$ is birational.

\end{lemma}

\begin{proof}

For a generic point $(l_0,l,l_1,l_a)\in(\mathbb{R}P^1)^4$ we construct a unique preimage under $\mathcal{B}$. Recall that in this text $\mathbb{R}P^1$ is seen as a real projectivisation of $\mathbb{C}$.

Let $\mathbb{C}\slash l$ be the set of all lines $L\subset\mathbb{C}$ parallel to $l.$ Denote by $x(L)$ the intersection of $L\in \mathbb{C}\slash l$ with $l_0$ and let $x(L)+y(L)$ be the intersection of $L$ with $l_1+1$. Define $\lambda(l_0,l,l_1) \subset\mathbb{C}$ as the locus of points of the form $z(L)=x(L)+ky(L).$ If $l_0,\ l_1$ and $l$ are generic then $\lambda(l_0,l,l_1) \subset\mathbb{C}$ is a line.

The lines $\lambda(l_0,l,l_1)$ intersects $l_a+a$ at $z(\hat L)$ for a unique $\hat L\in\mathbb{C}\slash l.$ The preimage of $(l_0,l,l_1,l_a)$ under $\mathcal{B}$ is $(x(\hat L),y(\hat L))\in V.$\end{proof}

Now we analyze how this invertion procedure can fail if the genericity assumption on $(l_0,l,l_1,l_a)$ is droped. If $l_0=l=l_1\neq\mathbb{R}$ then $\lambda(l_0,l,l_1)$ is empty and $(l_0,l,l_1,l_a)$ has no preimages under $\mathcal{B}$. The point $(\mathbb{R},\mathbb{R},\mathbb{R},l_a)$ has infinetly many preimages for any $l_a\in\mathbb{R}P^1$ and, therefore, is a critical value of $\mathcal{B}$. Points of the form $(\mathbb{R},l,\mathbb{R},l_a)$ for $l\neq\mathbb{R}$ have no preimages under the rolled coamoeba map.

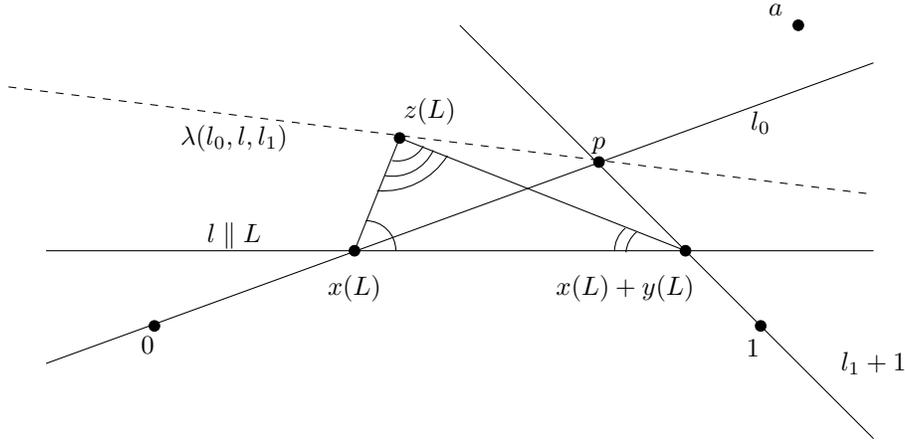
\begin{figure}[h]
\begin{tikzpicture}
\draw (-5.5,0) -- (5.5,0);
\draw (-5.5,-1.5) -- (5.5,2.5);
\draw (0,3) -- (5.5,-2.5);

\draw [dashed](-6,2.18) -- (5.5,0.75);
\draw(-3,1.5)node{$\lambda(l_0,l,l_1)$};

\draw (-1.4,0) -- (-0.8,1.5) -- (3,0);
\draw(-1.4,-0.5) node {$x(L)$};
\draw(2.2,-0.5) node {$x(L)+y(L)$};
\draw(-0.4,1.85) node {$z(L)$};
\filldraw  (1.85,1.4) node{$p$};

\draw(4,1.7) node {$l_0$};
\draw(5.5,-1.5) node {$l_1+1$};
\draw(-3,0.2) node {$l\parallel L$};
\draw(-4.15,-1.25) node {$0$};
\draw(3.9,-1.3) node {$1$};
\draw(4.2,3.2) node {$a$};

\filldraw  (-1.4,0) circle (2pt);
\filldraw  (4.5,3) circle (2pt);
\filldraw  (-4.06,-1) circle (2pt);

\filldraw  (4,-1) circle (2pt);

\filldraw  (3,0) circle (2pt);

\filldraw  (1.85,1.18) circle (2pt);

\filldraw  (-0.8,1.5) circle (2pt);

\draw (-0.85,0) arc (0:90:0.4cm);

\draw (-0.9,1.2) arc (-100:-30:0.4cm);
\draw (-1,1) arc (-100:-25:0.6cm);
\draw (-1.1,0.8) arc (-95:-32:1cm);

\draw (2.2,0.3) arc (130:180:0.4cm);
\draw (2.35,0.25) arc (130:175:0.4cm);

\end{tikzpicture}
\caption{The locus $\lambda(l_0,l,l_1)=\{z(L)=x(L)+ky(L): L \in\mathbb{C}\slash l\}$. }
\label{lambda}
\end{figure}

If $l_0$ and $l_1+1$ intersect at a single point $p\in\mathbb{C}$ then the locus $\lambda(l_0,l,l_1)$ is either $\{p\}$ or a line. The former happens only if $p$ belongs to the circle circumscribing the triangle with vertices $0$, $1$ and $k.$ If $\lambda(l_0,l,l_1)$ has no intersections with $l_a+a$ then $(l_0,l,l_1,l_a)$ is not a value of $\mathcal{B}.$ If $\lambda(l_0,l,l_1)\subset l_a+a$ then $(l_0,l,l_1,l_a)$ has infinetly many preimages.

Finally, for $l\neq l_0=l_1\neq \mathbb{R}$ the locus $\lambda(l_0,l,l_1)$ is a line parallel to $l_0.$ If this line intersects $l_a+a$ at a unique point then $(l_0,l,l_1,l_a)$ is a regular value, if $\lambda(l_0,l,l_1)\cap (l_a+a)=\emptyset,$ it is not a value, and if $\lambda(l_0,l,l_1)=l_a+a,$ it is a critical value (with infinetly many preimages).

\begin{corollary}
For $(l_0,l,l_1,l_a)\in\mathcal{B}_Z$ the lines $l_0,$ $l_1+1$ and $l_a$ intersect at a single point or parallel.
\end{corollary}

\begin{remark}In the complete analogy with what we just did with respect to a pair of points $0$ and $1$ (by means of the locus $\lambda(l_0,l,l_a)$ and its study), the inversion construction and the analysis of its degeneration can be performed with respect to a pair of $1$ and $a$ or a pair of $a$ and $0.$ 
\label{rem_sym}
\end{remark}

\section*{Proof of Theorems \ref{thm_five} and \ref{thm_fiber}}
Lets construct the compactification of $Z$ as the total space of a circle bundle over $\mathcal{B}_Z.$ Note that the fibers over $c\in\mathcal{B}_Z$ of the projection $\mathcal{B}|_\mathcal{Z}$ consist of linkage configurations with triangles spanned by vertices $x,x+y,x+ky\in \mathbb{R}P^2,$ where all such triangles are related through homotety with center at $p=\phi(c)$ if $p\in\mathbb{C}$ or through the parallel transport along $p\in\mathbb{R}P^1$ otherwise. In particular, the fiber is a projective line of triangles besides at most 5 special triangles, that correspond to degenerate linkage configurations when either
\begin{itemize}
\item $x=0$,
\item $y=0,$ i.e. the triangle degenerates to $p$,
\item $x+y=1$,
\item $x+ky=a$,
\item $x$, $x+y,$ or $x+ky\in \mathbb{R}P^1,$ i.e. the triangle has points at infinity.
\end{itemize}
\noindent We denote the corresponding sections by $s_0,\ s_p,\ s_1,\ s_a$ and $s_\infty.$ While going along the circle-fiber, crossing one of the first four sections changes the argument of the corresponding coordinate to the opposite one. Passing through $s_\infty$ changes arguments of all four coordinates.

\section*{Proof of Theorem \ref{thm_ten}}

On Figure \ref{sectionsC} we plot the loci where the five sections $s_0,\ s_p,\ s_1,\ s_a$ and $s_\infty$ pairwise coincide on $\mathcal{B}_Z$ (represented as $\mathbb{R}P^2$ blown-up at $0,\ 1,\ a,$ and $d$).

\begin{remark}
Every pair of sections coincides along a topological circle and very triple has no common intersections.
\end{remark}

Therefore, there exist a unique compactification $\hat{\mathcal{C}}_Z$ of $\mathcal{C}_Z$ to which the the projection  $\mathcal{C}_Z\rightarrow\mathcal{B}$ extends as an unramified five-fold covering. The fibers are represented by five intervals (some of which can be degenerate if a pair of sections concides over a given point) between the five sections of $\hat{Z}\rightarrow\mathcal{B}_Z$.
\begin{figure}[h!]
\includegraphics{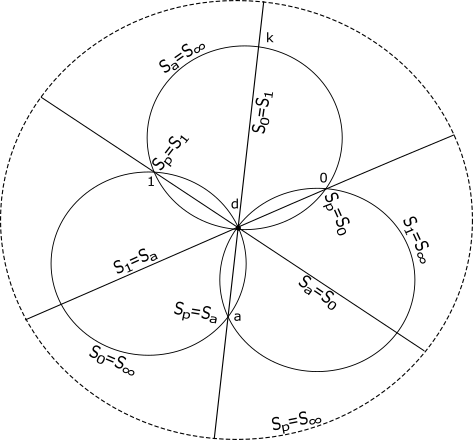}
\caption{ Ten loci where pairs of the five sections coincide. Circles and lines are constructed as in Figure \ref{pointD}. Note that $s_p$ coincides with $s_0$, $s_1$, $s_a $ along the exceptional divisors at $0$, $1$, $a$ respectively and with $s_\infty$ along the line at infinity.}
\label{sectionsC}
\end{figure}

Computing monodromies we deduce that  $\hat{\mathcal{C}}_Z$ is connected (see Figure \ref{circles} for more details). Thus,  $\hat{\mathcal{C}}_Z$ is homeomorphic to $\mathbb{R}P^2$ blown-up at sixteen points by classification of topological surfaces.
\begin{figure}
\includegraphics{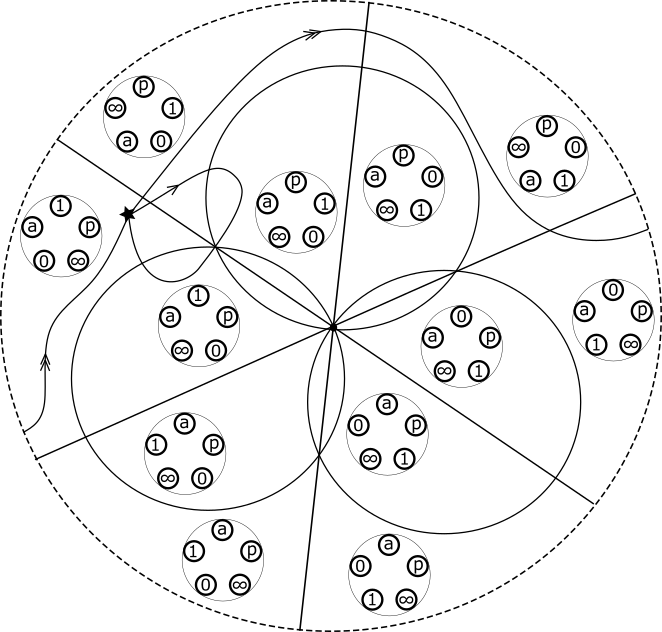}
\caption{A schematic diagram for the ciclic orders of  sections $s_0,\ s_p,\ s_1,\ s_a$ and $s_\infty$ derived from and compatible with Figure \ref{sectionsC}. Here we abuse the notation for a section replacing $s_\bullet$ by $\bullet$ inscribed into a small circle. The star represents the base point for two loops on $\mathcal{B}_Z$. A group generated by monodromies along these loops acts transitively on the fiber of the covering $\hat{\mathcal{C}}_Z\rightarrow\mathcal{B}_Z$ at the base point, implying that $\hat{\mathcal{C}}_Z$ is connected. }
\label{circles}
\end{figure}
\newpage
\section*{References}
[1] Natal'ya Bushueva, Avgust Tsikh, ``On amoebas of algebraic sets of higher codimension'', Proceedings of the Steklov Institute of Mathematics 279.1, 2012 \hfill\break
[2] Grigory Mikhalkin, ``Amoebas of half-dimensional varieties'', Analysis Meets Geometry. Birkhäuser, Cham, 2017 \hfill\break
[3] Mounir Nisse, Mikael Passare, ``Amoebas and coamoebas of linear spaces'', Analysis Meets Geometry. Birkhäuser, Cham, 2017\hfill\break
\end{document}